\documentclass[12pt]{amsart}
\usepackage{amsthm,amsfonts, amssymb, amscd}

\newtheorem{thm}{Theorem}

\newtheorem{lemma-definition}[thm]{Lemma-Definition}

\newtheorem{cor}[thm]{Corollary}
\theoremstyle{remark}

\begin{document}
\numberwithin{thm}{section}
\title[Compactifying algebraic spaces]
{Compactifying normal algebraic spaces}
\author{Dan Edidin}
\address{
Department of Mathematics,
University of Missouri,
Columbia, MO 65211
}
\email{edidin@math.missouri.edu}

\begin{abstract}
The author wrote this note after being asked 
about the existence of compactifications of algebraic spaces.
Subsequent to posting the article to the math arXiv, the author
learned from Yutakaa Matsuura that the results of this paper
had been proved by Raoult in his 1971 paper \cite{Rao:71} using
the same techniques.
Since Raoult's article may be unknown to those working
in the field, the author is keeping this preprint on the arXiv
server. However, he makes no claim of originality.
\end{abstract}

\maketitle
A classic theorem of Nagata \cite{Nag:62, Nag:63} states than any variety
may be embedded into a complete scheme.
Nagata's proof was translated to the language of schemes
by Deligne\footnote{See \cite{Con:07} or \cite{Voj:07} for an exposition
of Deligne's argument. An independent scheme-theoretic proof
was given by L\"utkebohmert \cite{Lut:93}.}.  It is now known
that any separated scheme of finite type over a quasi-compact and quasi-separated
base scheme may be embedded as an open subscheme in a scheme which is proper over
the base \cite[Theorem 4.1]{Con:07}. 

Nagata's completion result is an important technical tool in a
number of contexts; for example it is used in the construction
of higher direct images in \'etale cohomology with compact support (\cite{Mil:80}).
A natural problem
is to determine whether separated algebraic spaces and, more generally,
separated Deligne-Mumford stacks admit compactifications. 

Unfortunately
the essential idea used in all proofs of  Nagata's theorem 
is not available in the category of algebraic 
spaces. The point is that any scheme admits a cover by Zariski open
subschemes which are quasi-projective over the ground scheme.
The open sets in the cover admit obvious compactifications
and  a global compactification may be constructed via a delicate
gluing process. Because algebraic spaces and Deligne-Mumford
stacks are only \'etale locally quasi-projective schemes, it is not
apparent how to carry over the compactification strategy used
for schemes.

The purpose of this note is to show that using Galois descent it is possible
to give an easy proof that normal algebraic spaces admit compactifications.

\begin{thm} \label{thm.easycompact}
Let $X$ be a normal algebraic space which is separated and of finite type
over a Noetherian scheme $S$. Then there is an algebraic
space $\overline{X}$ which is proper over $S$ and contains $X$
as a dense open subspace.
\end{thm}
\begin{proof}
Since $X$ is a normal algebraic space it is generically a normal scheme
and hence has a function field $K(X)$. Also, $X$ is Noetherian as it is
of finite type over the Noetherian scheme $S$.
By \cite[Corollaire 16.6.2 ]{LMB:00} the normal algebraic space $X$ is a geometric quotient
$Y/\Gamma$ where $Y$ is a normal scheme and $\Gamma = Gal(K(Y)/K(X))$ is finite.
By Nagata's theorem the scheme $Y$ has a compactification  $\overline{Y}$ which 
is proper over $S$ (and hence Noetherian).
The following construction (which we learned from Sumihiro's paper \cite{Sum:74})
allows us to replace $\overline{Y}$ by a $\Gamma$-equivariant compactification.

Enumerate the elements of $\Gamma$ as $\gamma_1 = e, \ldots ,
\gamma_n$ and define a pairing $l \colon [1,n] \times [1,n] \to [1,n]$
by $\gamma_i \gamma_j = \gamma_{l(i,j)}$.  Let $\Gamma$ act on the $n$-fold product over $S$,
$\overline{Y}^n$, by the rule $\gamma_i(y_1, \ldots , y_n) =
(y_{l(1,i)} , \ldots , y_{l(n,i)})$. 
With our chosen $\Gamma$ action the embedding $s \colon Y \to
\overline{Y}^n$,
$z \mapsto (\gamma_1 z, \ldots, \gamma_n z)$ is $\Gamma$-equivariant.
Let $\overline{W}$ be the closure of the scheme theoretic image of $s$.
Since $s$ is $\Gamma$-equivariant, there is an action of $\Gamma$ on $\overline{W}$
which extends the action on the dense open subscheme $Y$. 
Moreover, $\overline{W}$ is proper over $S$ (and hence Noetherian)
since it is a closed subscheme of the proper $S$-scheme $\overline{Y}^n$.

Let $\overline{X}$ be the geometric quotient of $\overline{W}$ by the action of
$\Gamma$ (such quotients always exist in the category of algebraic spaces).
Since $\overline{W} \to \overline{X}$ is finite, the algebraic space $\overline{X}$
is proper over $S$ and contains the quotient $X = Y/\Gamma$ as a dense open subspace.
\end{proof}

As a corollary we obtain the following compactification result
for separated morphisms of algebraic spaces over a Noetherian base scheme.
\begin{cor}
Let $X \stackrel{f} \to Y$ be a morphism of algebraic spaces
which is separated and of finite type.
Assume that $Y$ (and hence $X$)
is separated and of finite type over a Noetherian base scheme $S$ and that $X$ is normal.
Then there exists
an algebraic space $\overline{X}$ which is proper over $Y$ and contains $X$
as a dense open subspace.
\end{cor}
\begin{proof}
By Theorem \ref{thm.easycompact} we know that the normal algebraic space
$X$ has an $S$-compactification $\tilde{X}$. 
The map $\tilde{X} \times_S Y \to Y$ obtained by base change is proper
and contains $X \times_S Y$ as a dense open subspace. Let $\overline{X}$
be the closure (in the sense of algebraic spaces) of the graph of $f$ in 
$\tilde{X} \times_S Y$.
\end{proof}

{\bf Acknowledgement:} The author thanks Martin Olsson for suggesting to him the problem
of compactifying algebraic spaces.

\bibliographystyle{amsmath}
\def\cprime{$'$}

\end{document}